\documentclass[12pt]{amsart}

\usepackage{amsmath,amssymb,amsthm}

\usepackage{booktabs}
\usepackage{enumerate}
\usepackage{hyperref}

\renewcommand{\S}{\mathfrak{S}}
\newcommand{\Sk}{\S_{k+1}}
\newcommand{\C}{W} 
\newcommand{\dc}{dissimilarity characteristic}

\newcommand{\BB}{B}
\newcommand{\Bx}{\mkern4mu\overline{\mkern-4mu B\mkern-1mu}\mkern 1mu}
\DeclareMathOperator{\fix}{fix}

\newcommand{\cb}{{\circ\raise 1.6pt\hbox{\vrule width 3.7pt height .6pt}\bullet}}

\newtheorem{theorem}{Theorem}
\newtheorem{lemma}[theorem]{Lemma}
\newtheorem{proposition}[theorem]{Proposition}

\usepackage{tikz}
\usepackage{tkz-graph}
\usetikzlibrary{positioning,calc,decorations.markings}

\newlength{\dcsp}
\usepackage{microtype}

\usepackage[letterpaper,margin=1in]{geometry}

\title{Counting unlabeled $k$-trees}
\author{Andrew Gainer-Dewar}
\address{Department of Mathematics, Carleton College, Northfield, MN 55057}
\email{againerdewar@carleton.edu}

\author{Ira M. Gessel}
\address{Department of Mathematics, Brandeis University, MS 050, Waltham, MA 02453}
\email{gessel@brandeis.edu}
\thanks{This work was partially supported by a grant from the Simons Foundation (\#229238 to Ira Gessel).}

\date{2 May 2014}

\begin{document}
\maketitle

\begin{abstract}
  We count unlabeled $k$-trees by properly coloring them in $k+1$ colors and then counting orbits of these colorings under the action of the symmetric group on the colors.
\end{abstract}

\section{Introduction}

\tikzset{
 graphnode/.style = {circle, align=center, inner sep=1pt, minimum size=3mm},
 blacknode/.style = {graphnode, draw=black, fill=black, text=white},
 whitenode/.style = {graphnode, draw=black, fill=white, text=black},
 graphedge/.style = {draw},
 invisible/.style = {draw = none},
 rootedge/.style = {graphedge, ultra thick},
 diredge/.style = {graphedge, ->},
 edgelabel/.style = {rectangle, fill=white, draw},
 every picture/.append style={node distance=1.5cm},
 ->-/.style={decoration={
     markings,
     mark=at position .5 with {\arrow{>}}},
   postaction={decorate}}
}

The class of $k$-trees may be defined recursively: a $k$-tree is either a complete graph on $k$ vertices or a graph obtained from a smaller $k$-tree by adjoining a new vertex together with $k$ edges connecting it to a $k$-clique. Thus a $1$-tree is an ordinary tree. Figure \ref{f-2-tree} shows a 2-tree.

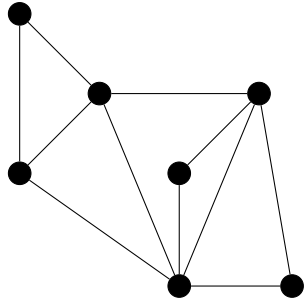
\begin{figure}[htbp] 
 \centering
 \begin{tikzpicture}
   \node [blacknode] (a) {};

   \node [blacknode, right of = a] (b) {};
   \path [graphedge] (a) edge (b);

   \node [blacknode, above of = a] (c) {};
   \path [graphedge] (a) edge (c);

   \node [blacknode, above right of = c] (d) {};
   \path [graphedge] (a) edge (d) (c) edge (d) (b) edge (d);

   \node [blacknode, above left of = c] (e) {};
   \path [graphedge] (a) edge (e) (d) edge (e);

   \node [blacknode, below left of = e] (f) {};
   \path [graphedge] (a) edge (f) (e) edge (f);

   \node [blacknode, above left of = e] (g) {};
   \path [graphedge] (e) edge (g) (f) edge (g);
 \end{tikzpicture}
 \caption{A 2-tree}
 \label{f-2-tree}
\end{figure}

Labeled $k$-trees may be counted by extensions of the same methods that can be used to count labeled trees. (See, for example, Beineke and Pippert \cite{bp}, Moon \cite{moon}, and Foata \cite{foata}.) However, counting unlabeled $k$-trees is considerably more difficult.

Harary and Palmer \cite{h-p} counted unlabeled 2-trees in 1968 (see also
Harary and Palmer \cite[section 3.5]{hpbook}) and unlabeled 2-trees were counted in a different way, using the theory of combinatorial species, by Fowler et al.\ in \cite{spec}. Many variations of 2-trees have also been counted \cite{solid, outerplanar, h-p,plane, k-gonal, k-gonal2}.

The enumeration of unlabeled $k$-trees for $k>2$ was a long-standing unsolved problem until the recent solution by Gainer-Dewar \cite{gainer}, also using the theory of combinatorial species.
We present here an alternative approach which results in a simpler description of the generating function for unlabeled $k$-trees; this is given in Theorem \ref{thm:main}.
The asymptotic growth of the number of $k$-trees has been analyzed by Drmota and Jin in \cite{drmota} using our results.

Table \ref{ta-k-trees}  gives the number $K_{n,k}$ of $k$-trees with $n+k$ vertices for small values of $n$ and $k$; larger tables can be found in \cite{gainer}.
The stability of these numbers for fixed $n$ as $k$ increases and a relation concerning those stable numbers will be explained in section \ref{s-coding}; these ``stable $k$-tree numbers'' are shown in the last row of the table.
These sequences are given to many more terms as \cite[A000055, A054581, A078792, A078793, A201702, A224917]{oeis} (for $k = 1, 2, 3, 4, 5$ and the stable $k$-tree numbers respectively).
\begin{table}[ht]
 \centering
 \begin{tabular}{l | *{10}r}
   \toprule
   $k \backslash n$ & 0 & 1 & 2 & 3 & 4 & 5 & 6 & 7 & 8 & 9 \\
   \midrule
   1 & 1 & 1 & 1 & 2 & 3 & 6 & 11 & 23 & 47 & 106 \\
   2 & 1 & 1 & 1 & 2 & 5 & 12 & 39 & 136 & 529 & 2171 \\
   3 & 1 & 1 & 1 & 2 & 5 & 15 & 58 & 275 & 1505 & 9003 \\
   4 & 1 & 1 & 1 & 2 & 5 & 15 & 64 & 331 & 2150 & 15817 \\
   5 & 1 & 1 & 1 & 2 & 5 & 15 & 64 & 342 & 2321 & 18578 \\
   \midrule
   ``stable'' & 1 & 1 & 1 & 2 & 5 & 15 & 64 & 342 & 2344 & 19137 \\
   \bottomrule
 \end{tabular}
 \caption{The number of $k$-trees with $n+k$ vertices}
 \label{ta-k-trees}
\end{table}

The usual approach to counting unlabeled trees or tree-like graphs consists of two steps. (See, for example, Bergeron, Labelle, and Leroux \cite[chapters 3 and 4]{bll}, and Harary and Palmer \cite[chapter 3]{hpbook}.)

First, one converts the problem to one of counting certain rootings of these graphs.
To do this, Otter \cite{otter} introduced the method of  ``dissimilarity characteristic theorems'' in counting  unlabeled trees. His theorem relates the numbers of orbits of vertices and edges under the automorphism group of the tree.   We will instead use a ``dissymmetry theorem'' in the style of those introduced by Leroux \cite{leroux} (see also Leroux and Miloudi \cite{l-m}); these describe isomorphisms of species and thus may be used to study labeled and unlabeled structures simultaneously.
(We will discuss  dissimilarity characteristic theorems for trees and 2-trees in sections \ref{sec:counttrees} and \ref{sec:count2trees}.)

Next, once the problem is reduced through a dissymmetry or dissimilarity characteristic theorem to one of counting rooted graphs, one finds recursive decompositions for the rooted graphs that give functional equations for their generating functions.

A dissymmetry theorem for $k$-trees is fairly straightforward, and it reduces the problem to counting $k$-trees rooted at a $(k+1)$-clique, $k$-trees rooted at a $k$-clique, and $k$-trees rooted at both a $(k+1)$-clique and a $k$-clique that it contains. (The \dc\ approach  to $k$-trees is more complicated.)

For classical (1-)trees, the decomposition step is easy: removing the root  of a vertex-rooted tree yields a set of trees, each rooted at the vertex which was adjacent to the original root, and similar decompositions are available for trees rooted in other ways.
However, for $k > 1$, this straightforward procedure is no longer sufficient.
For example, consider the two distinct edge-rooted $2$-trees of Figure \ref{f-edge-rooted}.
\begin{figure}[htbp] 
 \centering
 \begin{tikzpicture}
   \node [blacknode] (a) {};

   \node [blacknode, above left of = a] (b) {};
   \path [graphedge] (a) edge (b);

   \node [blacknode, above right of = a] (c) {};
   \path [graphedge] (a) edge (c);

   \node [blacknode, above right of = b] (d) {};
   \path [rootedge] (a) edge (d);
   \path [graphedge] (b) edge (d) (c) edge (d);

   \node [blacknode, above of = b] (e) {};
   \path [graphedge] (b) edge (e) (d) edge (e);

   \node [blacknode, above of = c] (f) {};
   \path [graphedge] (c) edge (f) (d) edge (f);
 \end{tikzpicture}
 \qquad
 \begin{tikzpicture}
   \node [blacknode] (a) {};

   \node [blacknode, above left of = a] (b) {};
   \path [graphedge] (a) edge (b);

   \node [blacknode, above right of = a] (c) {};
   \path [graphedge] (a) edge (c);

   \node [blacknode, above right of = b] (d) {};
   \path [rootedge] (a) edge (d);
   \path [graphedge] (b) edge (d) (c) edge (d);

   \node [blacknode, below of = b] (e) {};
   \path [graphedge] (b) edge (e) (a) edge (e);

   \node [blacknode, above of = c] (f) {};
   \path [graphedge] (c) edge (f) (d) edge (f);
 \end{tikzpicture}
 \caption{Two edge-rooted 2-trees}
 \label{f-edge-rooted}
\end{figure}
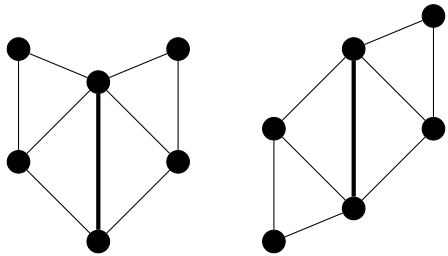

We can decompose each of these $2$-trees by first separating the triangles containing the root edge, and then from each of these triangles, removing the root edge and separating the two remaining edges, as shown in Figure \ref{f-decomp}.

\begin{figure}[htbp] 
 \centering
 \begin{tikzpicture}
   \node [invisible] (a) {};

   \node [invisible, above left = of a] (b) {};

   \node [invisible, above right = of a] (c) {};

   \node [invisible, above right = of b] (d) {};
   \path [rootedge, dashed] (a) edge (d);

   \node [blacknode, left = \dcsp of a] (a1) {};

   \node [blacknode, below = \dcsp of b] (b1) {};
   \path [rootedge] (a1) edge (b1);

   \node [blacknode, above = \dcsp of b] (b2) {};

   \node [blacknode, left = \dcsp of d] (d1) {};
   \path [rootedge] (b2) edge (d1);

   \node [blacknode, above = of b1] (e1) {};
   \path [graphedge] (b2) edge (e1) (d1) edge (e1);

   \node [blacknode, right = \dcsp of a] (a2) {};

   \node [blacknode, below = \dcsp of c] (c1) {};
   \path [rootedge] (a2) edge (c1);

   \node [blacknode, above = \dcsp of c] (c2) {};

   \node [blacknode, right = \dcsp of d] (d2) {};
   \path [rootedge] (c2) edge (d2);

   \node [blacknode, above = of c1] (e2) {};
   \path [graphedge] (c2) edge (e2) (d2) edge (e2);
 \end{tikzpicture}
 \qquad
 \begin{tikzpicture}
   \node [invisible] (a) {};

   \node [invisible, above left = of a] (b) {};

   \node [invisible, above right = of a] (c) {};

   \node [invisible, above right = of b] (d) {};
   \path [rootedge, dashed] (a) edge (d);

   \node [blacknode, left = \dcsp of a] (a1) {};

   \node [blacknode, below = \dcsp of b] (b1) {};
   \path [rootedge] (a1) edge (b1);

   \node [blacknode, above = \dcsp of b] (b2) {};

   \node [blacknode, left = \dcsp of d] (d1) {};
   \path [rootedge] (b2) edge (d1);

   \node [blacknode, below = of b2] (e1) {};
   \path [graphedge] (b1) edge (e1) (a1) edge (e1);

   \node [blacknode, right = \dcsp of a] (a2) {};

   \node [blacknode, below = \dcsp of c] (c1) {};
   \path [rootedge] (a2) edge (c1);

   \node [blacknode, above = \dcsp of c] (c2) {};

   \node [blacknode, right = \dcsp of d] (d2) {};
   \path [rootedge] (c2) edge (d2);

   \node [blacknode, above = of c1] (e2) {};
   \path [graphedge] (c2) edge (e2) (d2) edge (e2);
 \end{tikzpicture}
 \caption{Decomposition of  edge-rooted 2-trees}
 \label{f-decomp}
\end{figure}

We would like to recover the original $2$-trees from these collections of components.
However, it will not suffice (as in the case of $1$-trees) to treat each collection of components as simply a set; in particular, this would fail to distinguish between the two $2$-trees in the example.
In fact, each pair of 2-trees is  ordered (so the two decompositions in Figure~\ref{f-decomp} really are different), but switching the order of all pairs simultaneously does not change the 2-tree.

To deal with this problem, we may orient the root edge.
Extending the orientation of the root edge to an orientation of the triangles containing it orients the root edges of the component 2-trees, as shown in Figure  \ref{f-oriented}.

\begin{figure}[htbp] 
 \centering
 \begin{tikzpicture}
   \node [invisible] (a) {};

   \node [invisible, above left = of a] (b) {};

   \node [invisible, above right = of a] (c) {};

   \node [invisible, above right = of b] (d) {};
   \path [rootedge, diredge, dashed] (a) edge (d);

   \node [blacknode, left = \dcsp of a] (a1) {};

   \node [blacknode, below = \dcsp of b] (b1) {};
   \path [rootedge, diredge] (b1) edge (a1);

   \node [blacknode, above = \dcsp of b] (b2) {};

   \node [blacknode, left = \dcsp of d] (d1) {};
   \path [rootedge, diredge] (d1) edge (b2);

   \node [blacknode, above = of b1] (e1) {};
   \path [graphedge] (b2) edge (e1) (d1) edge (e1);

   \node [blacknode, right = \dcsp of a] (a2) {};

   \node [blacknode, below = \dcsp of c] (c1) {};
   \path [rootedge, diredge] (c1) edge (a2);

   \node [blacknode, above = \dcsp of c] (c2) {};

   \node [blacknode, right = \dcsp of d] (d2) {};
   \path [rootedge, diredge] (d2) edge (c2);

   \node [blacknode, above = of c1] (e2) {};
   \path [graphedge] (c2) edge (e2) (d2) edge (e2);
 \end{tikzpicture}

 \caption{Decomposition of  edge-rooted 2-trees}
 \label{f-oriented}
\end{figure}
Then 2-trees rooted at an oriented edge can be counted easily. The two-element group acts on  these oriented rooted 2-trees by reversing  the orientation, and 2-trees rooted at an unoriented edge are obtained as orbits under this group action.
This is essentially the approach taken by Fowler et al.~\cite{spec}. Gainer-Dewar \cite{gainer} took a similar approach to counting $k$-trees by cyclically orienting the $(k+1)$-cliques in a $k$-tree. This is significantly more complicated than in the case of 2-trees, since there is no simple analogue in this context of reversing the cyclic order of a triangle.

Here we take a somewhat different, though related approach. We color  the vertices of a $k$-tree in  $k+1$ colors, with all the vertices in each $(k+1)$-clique  colored differently. This coloring breaks the symmetry of the $k$-tree, allowing the decomposition to work. Then we take orbits under the symmetric group $\Sk$ acting on colors.

%

%

\section{Enumerative lemmas}

In this section we describe  two enumerative tools that we will need for counting $k$-trees. The reader may  skip this section for now and come back to it when it is needed.

Suppose that a finite group $G$ acts on a weighted set $S$. We do not require  $S$ to be finite, but we do require that the sum of the weights of all the elements of $S$ is well-defined as a formal power series, and we require that weights are constant on orbits. Thus we may define the weight of an orbit to be the weight of any of its elements.
For each $g\in G$ we denote by $\fix(g)$ the sum of the weights of the elements of $S$ fixed by $g$.
Then Burnside's lemma (also called the Cauchy-Frobenius theorem) asserts that the sum of the weights of the orbits of $G$ is equal to
\begin{equation}
 \label{e-burnside}
 \frac{1}{|G|}\sum_{g\in G}\fix(g).
\end{equation}

Since $\fix(g)$ depends only on the conjugacy class of $g$, the sum in \eqref{e-burnside} may be rewritten as a sum over conjugacy classes. We will be applying Burnside's lemma to the case of a symmetric group $\S_m$, where the conjugacy classes correspond to cycle types, which may be described by partitions of $m$: a permutation with $l_i$ cycles of length $i$ for each $i$ corresponds to the partition in which the multiplicity of $i$ as a part is  $l_i$.
The number of permutations of cycle type $\lambda=(1^{l_1} 2^{l_2}\cdots m^{l_m})$ is $m!/z_\lambda$, where
$z_\lambda=1^{l_1}l_1!\, 2^{l_2}l_2!\cdots m^{l_m} l_m!$. Using the notation $\lambda\vdash m$ to mean that $\lambda$ is a partition of  $m$, we may restate Burnside's lemma for symmetric groups in the following way.

\begin{lemma}
 \label{l-burnside}
 Let the symmetric group $\S_m$ act on the weighted set $S$ so that weights are constant on orbits. For each partition $\lambda$ of $m$ let $f_\lambda$ be the sum of the weights of the elements of $S$ fixed by a permutation of cycle type $\lambda$.
 Then the sum of the weights of the orbits of $S$ under $\S_m$ is
\(
  \sum_{\lambda\vdash m} f_\lambda/z_\lambda.
\)
\end{lemma}

For our second lemma, we return to the general case of a finite group $G$ acting on a weighted set $S$ (though we will only apply it to symmetric groups).  We now assume that each weight is a product of powers of variables. (In our application each weight will be a power of $x$.)

Let $M(S)$ be the set of multisets of elements of $S$. The action of $G$ on $S$ extends naturally to an action on $M(S)$. We define the weight of a multiset in $M(S)$ to be the product of the weights of its elements. As before, for any $g\in G$, we denote by $\fix(g)$ the sum of the weights of the elements of $S$ fixed by $g$. For any formal power series $u$ in the variables that occur in the weights, we denote by $p_n[u]$  the result of replacing each variable in $u$ by its $n$th power.

\begin{lemma}
 \label{l-ms}
 Let $g$ be an element of $G$. Then the sum of the weights of the elements of $M(S)$ fixed by $g$ is
 \begin{equation}
   \label{e-th1}
   \exp\biggl(\sum_{m=1}^\infty   \frac{p_m[\fix(g^{m})]}{m} \biggr).
 \end{equation}
\end{lemma}

\begin{proof}
 Let $F_g(S)$ be the sum of the weights of the elements of $M(S)$ fixed by $g$ and let $E_g(S)$ be the expression in \eqref{e-th1}.
 A multiset of elements of $S$ is fixed by $g$ if and only if it is a multiset union of orbits of $S$ under $g$ (i.e., orbits of the subgroup of $G$ generated by $g$). Thus  $F_g(S) = \prod_O F_g(O)$
 where $O$ runs over the  orbits of $S$ under $g$, and it is easy to check that $E_g(S)=\prod_O E_g(O)$.

 Thus it is sufficient to prove the lemma in the case in which $g$ acts transitively on $S$. We may assume without loss of generality that $g$ acts as an $n$-cycle  on the $n$-element set $S$, where every weight is $x$. Then $\fix(g^m) = nx$ if $n$ divides $m$ and  $\fix(g^m) = 0$  otherwise. In this case we have
 \begin{equation*}
   E_g(S) = \exp\biggl(\sum_{i=1}^\infty   \frac{p_{ni}[\fix(g^{ni})]}{ni} \biggr)
   = \exp\biggl(\sum_{i=1}^\infty   \frac{x^{ni}}{i} \biggr)
   =\frac{1}{1-x^n}=F_g(S).\qedhere
 \end{equation*}
\end{proof}


%
%

\section{Coding Trees}
\label{s-coding}

Following the   terminology introduced in \cite{gainer}, we call a $(k+1)$-clique in a $k$-tree a \emph{hedron} and we call a $k$-clique in a $k$-tree a \emph{front}. We may think of a $k$-tree as made up of hedra joined along fronts. A $k$-tree with $n$ hedra has $n+k$ vertices, so we will count $k$-trees by the number of hedra rather than the number of vertices.

A \emph{colored} $k$-tree is a $k$-tree in which the vertices are colored in the colors 1, 2, \dots, $k+1$ so that adjacent vertices are colored differently. Thus the $k+1$ vertices of any hedron are colored with  all $k+1$ colors, and the $k$ vertices of any front are colored in all but one of the colors. It is not hard to see that a coloring of a $k$-tree is determined by its restriction to any one of its hedra. (See Figure \ref{f-1}.)
\begin{figure}[htbp] 
 \centering
 \begin{tikzpicture}
   \node [whitenode] (a) {$1$};

   \node [whitenode, below right = of a] (b) {$2$};
   \path [graphedge] (a) edge (b);

   \node [whitenode, above right = of a] (c) {$2$};
   \path [graphedge] (a) edge (c);

   \node [whitenode, above right = of b] (d) {$3$};
   \path [graphedge] (a) edge (d) (b) edge (d) (c) edge (d);

   \node [whitenode, above = of c] (e) {$2$};
   \path [graphedge] (a) edge (e) (d) edge (e);

   \node [whitenode, right = of e] (f) {$1$};
   \path [graphedge] (e) edge (f) (d) edge (f);
 \end{tikzpicture}

 \caption{A colored $2$-tree}
 \label{f-1}
\end{figure}
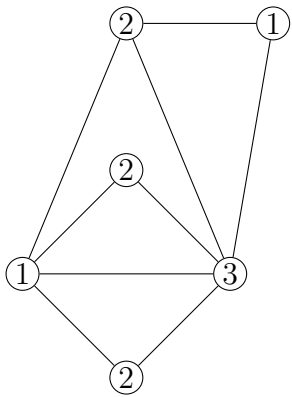
Then the symmetric group $\Sk$ acts on colored $k$-trees by permuting the colors, and  $k$-trees may be identified with orbits of colored $k$-trees under this action.

We will label the hedra of  $k$-trees. More precisely, a \emph{colored hedron-labeled $k$-tree} with hedron-label set $L$ is a colored $k$-tree together with a bijection from the set of its hedra to $L$. It is not hard to see that the only automorphism of a colored hedron-labeled $k$-tree that preserves hedra and colors is the  identity automorphism. (This is because any vertex  is determined by its color and a hedron that contains it.) Therefore we may ignore the labels on the vertices.

We will encode colored hedron-labeled $k$-trees by certain (ordinary) trees that we call \emph{coding trees} (or \emph{$k$-coding trees} if $k$ needs to be specified). Coding trees have two kinds of vertices: black vertices and colored vertices. Every edge joins a black vertex and a colored vertex. Each colored vertex is colored in one of the colors 1, 2, \dots, $k+1$ but is otherwise unlabeled, and the black vertices are labeled. (The colored vertices correspond to the white vertices of \cite{gainer}, which uses a related but different construction under the same name.) Every black vertex has $k+1$ colored neighbors, one of each color, but a colored vertex may have any number of neighbors.

To construct a coding tree from  a colored $k$-tree,  we first color each front of the $k$-tree with the unique color not assigned to any of its vertices. The coding tree then has a black vertex for each  hedron of the $k$-tree (from which it takes its label) and a colored vertex for each front of the $k$-tree (from which it takes its color), and a black vertex is adjacent to a colored vertex if and only if the corresponding hedron contains the corresponding front.

Figure \ref{f-two-2-trees} shows colored versions of the $2$-trees previously shown in Figure~\ref{f-edge-rooted}.
\begin{figure}[htbp]
 \centering
 \begin{tikzpicture}[node distance=2.5cm]
   \node [blacknode] (a) {};

   \node [blacknode, above left = of a] (b) {};
   \path [graphedge] (a) edge node [edgelabel] {$1$} (b);

   \node [blacknode, above right = of a] (c) {};
   \path [graphedge] (a) edge node [edgelabel] {$1$} (c);

   \node [blacknode, above right = of b] (d) {};
   \path [graphedge] (a) edge node [edgelabel] {$2$} (d) (b) edge node [edgelabel] {$3$} (d) (c) edge node [edgelabel] {$3$} (d);

   \node [blacknode, above = of b] (e) {};
   \path [graphedge] (b) edge node [edgelabel] {$1$} (e) (d) edge node [edgelabel] {$2$} (e);

   \node [blacknode, above = of c] (f) {};
   \path [graphedge] (c) edge node [edgelabel] {$1$} (f) (d) edge node [edgelabel] {$2$} (f);
 \end{tikzpicture}
 \qquad
 \begin{tikzpicture}[node distance=2.5cm]
   \node [blacknode] (a) {};

   \node [blacknode, above left = of a] (b) {};
   \path [graphedge] (a) edge node [edgelabel] {$1$} (b);

   \node [blacknode, above right = of a] (c) {};
   \path [graphedge] (a) edge node [edgelabel] {$1$} (c);

   \node [blacknode, above right = of b] (d) {};
   \path [graphedge] (a) edge node [edgelabel] {$2$} (d) (b) edge node [edgelabel] {$3$} (d) (c) edge node [edgelabel] {$3$} (d);

   \node [blacknode, below = of b] (e) {};
   \path [graphedge] (b) edge node [edgelabel] {$3$} (e) (a) edge node [edgelabel] {$2$} (e);

   \node [blacknode, above = of c] (f) {};
   \path [graphedge] (c) edge node [edgelabel] {$1$} (f) (d) edge node [edgelabel] {$2$} (f);
 \end{tikzpicture}

 \caption{Two colored $2$-trees}
 \label{f-two-2-trees}
\end{figure}
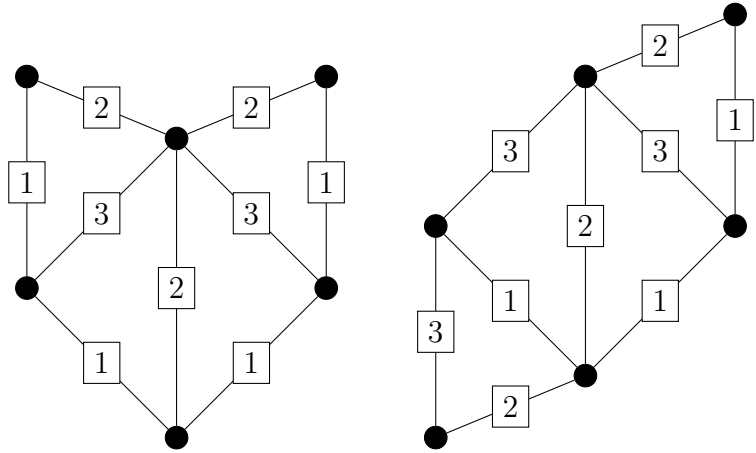
Their associated $2$-coding trees may be seen to be distinct in Figure \ref{f-two-coding}.
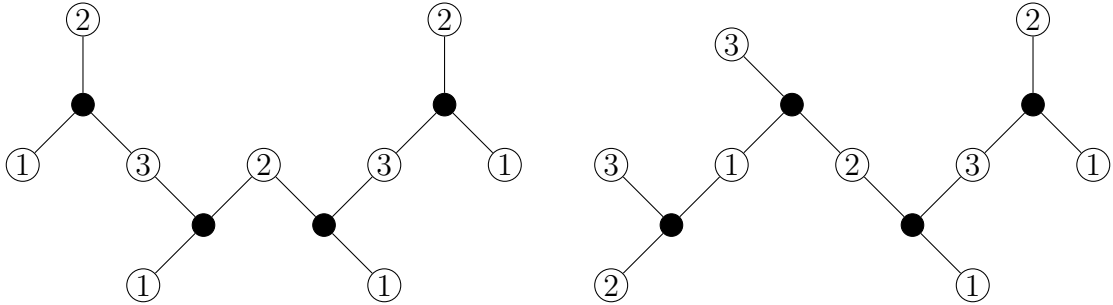
\begin{figure}[htbp]
 \centering
 \begin{tikzpicture}[node distance=.75cm]
   \node [whitenode] (a) {$2$};

   \node [blacknode, below left = of a] (1) {};
   \path [graphedge] (a) edge (1);

   \node [whitenode, below left = of 1] (b) {$1$};
   \path [graphedge] (b) edge (1);

   \node [whitenode, above left = of 1] (c) {$3$};
   \path [graphedge] (c) edge (1);

   \node [blacknode, above left = of c] (2) {};
   \path [graphedge] (c) edge (2);

   \node [whitenode, below left = of 2] (d) {$1$};
   \path [graphedge] (d) edge (2);

   \node [whitenode, above = of 2] (e) {$2$};
   \path [graphedge] (e) edge (2);

   \node [blacknode, below right = of a] (3) {};
   \path [graphedge] (a) edge (3);

   \node [whitenode, below right = of 3] (f) {$1$};
   \path [graphedge] (f) edge (3);

   \node [whitenode, above right = of 3] (g) {$3$};
   \path [graphedge] (g) edge (3);

   \node [blacknode, above right = of g] (4) {};
   \path [graphedge] (g) edge (4);

   \node [whitenode, below right = of 4] (h) {$1$};
   \path [graphedge] (h) edge (4);

   \node [whitenode, above = of 4] (i) {$2$};
   \path [graphedge] (i) edge (4);
 \end{tikzpicture}
 \qquad
 \begin{tikzpicture}[node distance=.75cm]
   \node [whitenode] (a) {$2$};

   \node [blacknode, above left = of a] (1) {};
   \path [graphedge] (a) edge (1);

   \node [whitenode, below left = of 1] (b) {$1$};
   \path [graphedge] (b) edge (1);

   \node [whitenode, above left = of 1] (c) {$3$};
   \path [graphedge] (c) edge (1);

   \node [blacknode, below left = of b] (2) {};
   \path [graphedge] (b) edge (2);

   \node [whitenode, below left = of 2] (d) {$2$};
   \path [graphedge] (d) edge (2);

   \node [whitenode, above left = of 2] (e) {$3$};
   \path [graphedge] (e) edge (2);

   \node [blacknode, below right = of a] (3) {};
   \path [graphedge] (a) edge (3);

   \node [whitenode, below right = of 3] (f) {$1$};
   \path [graphedge] (f) edge (3);

   \node [whitenode, above right = of 3] (g) {$3$};
   \path [graphedge] (g) edge (3);

   \node [blacknode, above right = of g] (4) {};
   \path [graphedge] (g) edge (4);

   \node [whitenode, below right = of 4] (h) {$1$};
   \path [graphedge] (h) edge (4);

   \node [whitenode, above = of 4] (i) {$2$};
   \path [graphedge] (i) edge (4);
 \end{tikzpicture}

 \caption{The corresponding coding trees}
 \label{f-two-coding}
\end{figure}

It is not hard to see that this encoding is a bijection from colored hedron-labeled $k$-trees to $k$-coding trees; the proof is essentially identical to that of Theorem 3.4 of~\cite{gainer}.

The symmetric group $\S_n$ acts on colored hedron-labeled  $k$-trees with hedron-label set $[n]$ by permuting the labels of the hedra, and $\S_n$ acts on $k$-coding trees with black vertex set $[n]$ by permuting the labels of the black vertices, and these actions are compatible with the encoding of   colored $k$-trees as coding trees.
The orbits under the action of $\S_n$ are \emph{unlabeled}  colored $k$-trees and coding trees.
The symmetric group $\Sk$ acts compatibly on colored $k$-trees and on coding trees by permuting the colors. These  actions of $\S_n$ and $\Sk$ commute, so $\Sk$ acts compatibly\footnote{This `equivariance' of the action of $\Sk$ can also be exploited for species-theoretic analysis; cf.~\cite{gainer}.} on unlabeled colored $k$-trees and on coding trees, and the orbits of unlabeled colored $k$-trees, which are simply unlabeled $k$-trees, are in bijection with orbits of unlabeled coding trees under the action of $\Sk$. Our goal then is to count orbits under $\Sk$ of unlabeled coding trees, which we call \emph{color-orbits} of coding trees.


Before proceeding with the enumeration, we show how coding trees give a simple way to explain the stability property for the numbers $K_{n,k}$ apparent in Table \ref{ta-k-trees}.

\begin{proposition}
\label{p-stab}
 Let $K_{n,k}$ be the number of $k$-trees with $n+k$ vertices.
 Then $K_{n,k-1}=K_{n,k}$ for $k\ge n-1$.
\end{proposition}
\begin{proof}
 We first note that $K_{n,k}$ is the number of  $k$-trees with $n$ hedra, and is thus the number of color-orbits of coding trees with $n$ black vertices.

 {}From any coding tree we may obtain a ``pruned coding tree" by deleting its leaves, all of which are colored vertices. In a pruned $k$-coding tree every leaf is black and  every black vertex has at most $k+1$ colored neighbors, whose colors are distinct integers from 1 to $k+1$.

 {}From a pruned $k$-coding tree we can recover the original $k$-coding tree by adding a leaf of color $i$ adjacent to every black vertex with no neighbor of color $i$, for $i$ from~1 to $k+1$. But every pruned $(k-1)$-coding tree is also a pruned $k$-coding tree,  so there is an injection from color-orbits of $(k-1)$-coding trees to color-orbits of $k$-coding trees.
 The only color-orbits of $k$-coding trees not in the image of this injection will be those corresponding to pruned coding trees with at least $k+1$ different colors. In order for a pruned coding tree to have $k+1$ different colors, it must have at least $k+1$ colored vertices. Now suppose that a pruned coding tree has $c$ colored vertices, where $c\ge k+1$. Since the colored vertices of a pruned coding tree are not leaves, each has degree at least 2, and since every edge of the tree is incident with one colored vertex, the tree has at least $2c$ edges, and therefore at least $2c+1$ vertices. Since $c$ vertices are colored, at least $c+1\ge k+2$ are black. Thus $K_{n,k-1}=K_{n,k}$ unless $n\ge k+2$.
\end{proof}

It seems unlikely that there is any simple formula for the ``stable $k$-tree numbers" 1, 1, 1, 2, 5, 15, 64, 342, 2344, \dots, which are listed in \cite[A224917]{oeis}.
It seems to be a coincidence that the second through sixth terms of this sequence are the first five Bell numbers.

We also note that coding trees  explain the differences between the ``stable $k$-tree numbers'' and the final non-stable number in the columns of Table \ref{ta-k-trees}.
\begin{proposition}
For $n\ge 4$ we have
\[ K_{n,n-2} - K_{n,n-3} = K_{n-1,1}. \]
\end{proposition}
\begin{proof}

In the terminology of the proof of Proposition \ref{p-stab}, $K_{n, n-2} - K_{n, n-3}$ counts color-orbits of pruned $(n-2)$-coding trees with $n$ black vertices in which there are colored vertices of all colors from 1 to $n-1$.

In the proof of Proposition \ref{p-stab} we showed that a pruned coding tree with $c$ colored vertices must have at least $c+1$ black vertices, with equality if and only if every colored vertex has degree 2. It follows that the pruned coding trees whose color-orbits are counted by  $K_{n, n-2} - K_{n, n-3}$ have exactly $n-1$ colored vertices,  all of degree 2, and all of different colors.

Since the colored vertices all have different colors, we can ignore colors in constructing the color-orbits, and since each colored vertex has degree two, we can merge  every colored vertex with its two incident edges to obtain an ordinary unlabeled tree on the $n$ black vertices; the number of such objects is $K_{n-1, 1}$ and
the desired equality follows.
\end{proof}

\section{A dissymmetry theorem}

We now prove a dissymmetry theorem that reduces the problem of counting unlabeled coding trees to that of counting rooted unlabeled coding trees.
\begin{lemma}
 \label{l-diss}
 Let $n$ be a positive integer and let $\C_n$ be the set of coding trees with black vertex set $[n]$.  Let $\C_n^\circ$ be the set of rooted coding trees obtained by rooting  a tree in $\C_n$ at a colored vertex, let $\C_n^\bullet$ be the set of rooted coding trees obtained by rooting  a tree in $\C_n$ at a black vertex, and let $\C_n^\cb$ be the set of rooted coding trees obtained by rooting  a tree in $\C_n$ at an edge.
 Then there is a bijection $\Theta$ from $\C_n^\circ\cup \C_n^\bullet$ to $\C_n\cup \C_n^\cb$ that commutes with the actions of $\S_n$ on vertex labels and  of $\Sk$ on colors.
\end{lemma}

\begin{proof}
 Every coding tree has a unique center vertex, either black or colored, which is the midpoint of every longest path in the tree, and the center is fixed by both group actions.
 Let $T$ be a rooted tree in $\C_n^\circ\cup \C_n^\bullet$. If $T$ is rooted at its center then we define $\Theta(T)$
 to be the underlying unrooted tree of $T$. Otherwise, there is a unique path from the root $r$ of $T$ to the center, and we take $\Theta(T)$ to be the underlying tree of $T$ rooted at the first edge on the path from $r$ to the center. It is easily seen that $\Theta$ is a bijection that commutes with the actions of $\S_n$ and $\Sk$.
\end{proof}

Since the bijection $\Theta$ of Lemma \ref{l-diss} commutes with the action of $\S_n$ on vertices, it is well-defined on unlabeled coding trees, and its application to coding trees commutes with the action of $\Sk$, so it gives a corresponding bijection for color-orbits of coding trees, which are equivalent to $k$-trees.

\section{Counting unlabeled coding trees}

{}From now on we work with unlabeled (but colored) coding trees. We define the \emph{weight} of a coding tree with $n$ black vertices to be $x^n$; then the generating function for a set of trees is the sum of the weights of its elements.

We call a coding tree rooted at a black vertex a \emph{black-rooted tree} and we call a  coding tree rooted at a colored vertex a \emph{colored-rooted tree}. We call a colored-rooted tree with root of color $j$  a \emph{$j$-rooted tree}.

Our next lemma follows directly from Lemma \ref{l-diss}.

\begin{lemma}
 \label{l-diss2}
 Let $U$ be the generating function for \textup{(}unlabeled\textup{)} color-orbits of coding trees, let $\BB$ be the generating function for color-orbits of black-rooted trees,  let $C$ be the generating function for color-orbits of colored-rooted trees, and let $E$ be the generating function for color-orbits of coding trees rooted at an edge. Then
 \begin{equation*}
   U =  \BB +C -E.
 \end{equation*}
\end{lemma}


%

It is clear that  $\BB$, $C$, and $E$ may also be interpreted in terms of rooted unlabeled $k$-trees.

Our goal in the remainder of this section is to compute the generating functions $\BB$, $C$, and $E$ appearing in Lemma \ref{l-diss2}.

Let us introduce some notation. For each $\pi\in \Sk$  let $\BB_\pi=\BB_\pi(x)$ be the generating function for black-rooted trees  that are fixed by $\pi$. It is clear that  $\BB_\pi$ depends only on the cycle type of $\pi$, so
for any partition $\lambda$ of $k+1$ we may set  $\BB_\lambda=\BB_\pi$, where $\pi$ is a permutation of cycle type $\lambda$.

For our decompositions we will need to consider a variation of black-rooted  trees. If we delete the root from a $j$-rooted tree,  we obtain trees rooted at black vertices that are like black-rooted  trees, but in which the roots have neighbors of all colors except~$j$. We call these trees \emph{$j$-reduced black-rooted  trees}.
If a $j$-reduced black-rooted  tree is fixed by $\pi$ then $j$ must be a fixed point of $\pi$, but as long as $j$ is a fixed point of $\pi$, the generating function for $j$-reduced black-rooted trees  fixed by $\pi$ depends only on the cycle type of $\pi$, and is zero if $\pi$ has no fixed points. So for any permutation $\pi\in \Sk$, we define $\Bx_\pi=\Bx_\pi(x)$ to be the generating function for $j$-reduced black-rooted  trees fixed by $\pi$, where $j$ is any fixed point of $\pi$ (and $\Bx_\pi=0$ if $\pi$ has no fixed points). For any partition $\lambda$ of $k+1$, we may define $\Bx_\lambda$ to be $\Bx_\pi$ where $\pi$ is
a permutation of $[k+1]$ of cycle type $\lambda$.

Similarly, if a colored-rooted  tree with root of color $j$ is fixed by $\pi$ then $j$ must be a fixed point of $\pi$. For a permutation $\pi$ of $[k+1]$, we define $C_\pi=C_\pi(x)$ to be the generating function for colored-rooted  trees fixed by $\pi$ with root of color $j$, where $j$ is any fixed point of $\pi$ (and $C_\pi=0$ if $\pi$ has no fixed points). For any partition $\lambda$ of $k+1$, we  define $C_\lambda$ to be $C_\pi$ where $\pi$ is
a permutation of $[k+1]$  of cycle type~$\lambda$.
Note that $\Bx_\lambda$ and $C_\lambda$ are zero if 1 is not a part of $\lambda$.
It will be convenient in the subsequent discussion  to define $\Bx_\mu$ and $C_\mu$ for $\mu$ a partition of~$k$ by $\Bx_\mu=\Bx_\lambda$ and $C_\mu=C_\lambda$ where $\lambda$ is obtained from $\mu$ by adding an additional part~1.

If $\lambda$ is a partition, then by $\lambda^i$ we mean the cycle type of $\pi^i$, where $\pi$ is a permutation of cycle type $\lambda$. We note that each $p$-cycle of $\pi$ contributes $(p,i)$ cycles, each of length $p/(p,i)$ to $\pi^i$, where $(p,i)$ is the greatest common divisor of $p$ and $i$.

We may now state our final result. Although our formula for $U$ consists of functional equations for a number of different power series, it is fairly easy to compute the coefficients of these series by successive substitution.

\begin{theorem}
 \label{thm:main}
 The generating function $U$ for unlabeled $k$-trees is given by
 \begin{equation*}
   U = \BB+C-E,
 \end{equation*}
 where
 \begin{subequations}
   \begin{align}
     \BB &= \sum_{\lambda\vdash k+1} \BB_{\lambda}/z_\lambda, \label{e-BB1}\\
     C &= \sum_{\mu\vdash k} C_\mu/z_\mu,\label{e-C}\\
     E &= \sum_{\mu\vdash k} \Bx_\mu C_\mu/z_\mu,\label{e-E}\\
     \BB_\lambda &= x\prod_i C_{\lambda^i}(x^{i}),\label{e-BBlambda}\\
     \Bx_\mu &= x\prod_i C_{\mu^{i}}(x^{i}),\label{e-Bmu}\\
     C_\mu &=  \exp\biggl(\sum_{m=1}^\infty   \frac{\Bx_{\mu^m}(x^m)}{m} \biggr).\label{e-Cmu}
   \end{align}
 \end{subequations}

 In \eqref{e-BBlambda}, $\lambda$ is a partition of $k+1$ and in \eqref{e-Bmu} and \eqref{e-Cmu}, $\mu$ is a partition of $k$.
 In  the products in \eqref{e-BBlambda} and \eqref{e-Bmu},
 $i$ runs through the parts of $\lambda$ and $\mu$ with multiplicities; i.e., if $i$ occurs $m$ times as a part then $i$ is taken $m$ times in the product.

 %

 %
\end{theorem}


%
%
%
\begin{proof}
 Formula \eqref{e-BB1} follows directly from Lemma \ref{l-burnside} (Burnside's lemma).

 The generating function $C$ for color-orbits under $\Sk$ of colored-rooted  trees is the same as the generating function  for  color-orbits under the action of $\S_k$, permuting the colors 1 through $k$, of $k+1$-rooted trees, since every color-orbit of colored-rooted trees contains a $k+1$-rooted tree.
 Then  \eqref{e-C} follows from Lemma \ref{l-burnside}.

 Similarly, the generating function $E$ for  color-orbits under $\Sk$ of coding trees rooted at an edge is the same as the generating function  for  color-orbits under the action of $\S_k$, permuting the colors 1 through $k$, of coding trees rooted at an edge incident with a vertex of color $k+1$.
 Removing the root edge from such a tree leaves a $k+1$-rooted  tree together with a $k+1$-reduced black-rooted  tree.
 Thus, if $\pi\in \Sk$ fixes $k+1$, the generating function for such pairs fixed by $\pi$ is $C_\pi \Bx_\pi$, so \eqref{e-E} follows.

 Next, for $\pi\in \Sk$ we find an equation for $\BB_\pi$, which counts black-rooted trees fixed by~$\pi$.
 The root of such a tree has $k+1$ children, one of each of the colors from 1 to $k+1$.
 If we delete the root, we are left with trees $T_1,\dots, T_{k+1}$, where tree $T_j$ is rooted at a vertex of color~$j$.
 Now suppose that $j$ is in a cycle of $\pi$ of length $i$.
 Then the orbit of $T_j$ under $\pi$ consists of  $T_j, T_{\pi(j)} =\pi(T_j), \dots, T_{\pi^{i-1}(j)}=\pi^{i-1}(T_j)$, and we must have $\pi^i(T_j) = T_j$.
 Thus to determine a black-rooted tree fixed by $\pi$, we choose from each cycle of $\pi$ an arbitrary element $j$, and take $T_j$ to be a $j$-rooted tree that is fixed by $\pi^i$, where $i$ is the length of the cycle of $\pi$ containing $j$.
 Then $T_{\pi(j)}, \dots, T_{\pi^{i-1}(j)}$ are determined and all have the same weight as $T_j$.
 The generating function for $j$-rooted trees fixed by $\pi^i$ is $C_{\pi^i}(x)$ (independently of the choice of $j$), so the contribution to $\BB_\pi$ from a cycle of  $\pi$ of length $i$ is $C_{\pi^i}(x^i)$.
 Thus
 \begin{equation}
   \label{e-BB}
   \BB_\pi=x\prod_c C_{\pi^{|c|}}(x^{|c|})
 \end{equation}
 where $c$ runs over the cycles of $\pi$ and $|c|$ is the size of the cycle $c$. Thus  \eqref{e-BBlambda} follows, and a similar argument gives
 \eqref{e-Bmu}.

 %

 %

 Next we need to find a formula for $C_\pi$. Since $C_\pi=0$ if $\pi$ has no fixed points, we may assume without loss of generality that $k+1$ is a fixed point of~$\pi$. Suppose that $T$ is a $k+1$-rooted tree that is fixed by the permutation~$\pi$. Removing the root from $T$ leaves a multiset of $k+1$-reduced black-rooted trees that is fixed by $\pi$. Thus $C_\pi$ is the generating function for these multisets, and applying Lemma~\ref{l-ms} gives
   \begin{equation*}
   C_\pi =  \exp\biggl(\sum_{m=1}^\infty   \frac{\Bx_{\pi^m}(x^m)}{m} \biggr),
 \end{equation*}
 and \eqref{e-Cmu} follows.
\end{proof}

%
%
%


\section{Counting trees}
\label{sec:counttrees}
Here we consider in detail the case $k=1$, where we are counting ordinary trees by the number of edges. There are only two elements of $\S_2$, and two partitions of $k+1=2$, only one of which has 1 as a part.

We have
{\allowdisplaybreaks
 \begin{align*}
   \BB_{1,1} &= xC_1(x)^2\\
   \BB_{2} &= xC_1(x^2) \\
   \Bx_1 &= xC_1(x)\\
   C_1  &=   \exp\biggl(\sum_{m=1}^\infty   \frac{\Bx_{1}(x^m)}{m} \biggr)\\
   \BB &= \frac{x}{2} \bigl(C_1(x)^2+C_1(x^2)\bigr)\\
   C &= C_1(x)\\
   E &= xC_1(x)^2.
 \end{align*}
}
Here everything can be expressed in terms of $C_1$, which is the generating function for vertex-rooted unlabeled trees by the number of edges, and we can simplify the result to
\begin{subequations}
 \begin{align}
   C_1 &=  \exp\biggl(\sum_{m=1}^\infty   \frac{x^m C_{1}(x^m)}{m} \biggr), \label{e-C1} \\
   U &= C_1(x) -\frac x2\bigl(C_1(x)^2-C_1(x^2)\bigr),\label{e-U1}
 \end{align}
\end{subequations}
which is equivalent to the well-known formula found by Otter \cite{otter}. (Unlabeled trees had been counted earlier using different methods by Cayley \cite{cayley1, cayley2}.)

We sketch here a version of Otter's dissimilarity characteristic interpretation of \eqref{e-U1}, as in the next section we will discuss a related, but more complicated, dissimilarity characteristic formula for 2-trees.




We define a \emph{symmetry edge} of a tree to be an edge $e$ such that some automorphism of the tree reverses the endpoints of $e$. We say that two vertices $u$ and $v$ of a tree are \emph{similar} if there is some automorphism of the tree that takes $u$ to $v$, and similar edges are defined analogously. Otter's \dc\ theorem is the following.

\begin{lemma}
 \label{l-diss1}
 The number of dissimilar vertices of a tree \textup{(}i.e., the number of equivalence classes under similarity\textup{)} is one more than the number of dissimilar edges that are not symmetry edges.
\end{lemma}
\begin{proof}
 Let  $T$ be a tree. Then $T$ has  a center which is either a vertex or an edge and is fixed by every automorphism of $T$. To each vertex $v$ that is neither a center vertex nor an endpoint of a center edge we associate the first edge on the path from $v$ to the center. This is an equivariant bijection (with respect to the automorphism group of the tree) from  the vertices of $T$ to the edges of~$T$ other than center vertices, center edges, and vertices of center edges, and so among the paired vertices and edges there are as many dissimilar vertices as edges. Then  the unpaired vertices and edges consist of one of the following: (i) a center vertex;
 (ii) a center edge which is not a symmetry edge, and its two dissimilar endpoints;
 (iii) a  center edge which is a symmetry edge, and its two similar endpoints.
 In each case, among the unpaired vertices and edges the number of dissimilar vertices is one more than the number of dissimilar non-symmetry edges.
\end{proof}

As a consequence, for any unlabeled tree, the number of dissimilar ways to root it at a vertex is one more than the number of dissimilar ways to root it at a non-symmetry edge. The generating function for unlabeled rooted trees is $C_1(x)$ and the generating function for unlabeled trees rooted at a non-symmetry edge is $\tfrac12x(C_1(x)^2 - C_1(x^2))$, so their difference counts every (unrooted) tree once.

\section{Counting 2-trees}
\label{sec:count2trees}
Next we look in detail at the case $k=2$.
Here we have
\begin{align*}
 \BB_{1,1,1} &=  xC_{1,1}^3 \\
 \BB_{2,1} &=  xC_{2}(x) C_{1,1}(x^2)\\
 \BB_{3}  &=  xC_{1,1}(x^3)\\
 \Bx_{1,1}  &= xC_{1,1}(x)^2  \\
 \Bx_{2}  &=  xC_{1,1}(x^2) \\
 C_{1,1} &=  \exp\biggl(\sum_{m=1}^\infty \frac{\Bx_{1,1}(x^m)}{m} \biggr) \\
 C_2  &= \exp\biggl(\sum_{\text{$m$ odd}} \frac{\Bx_{2}(x^m)}{m}
 + \sum_{\text{$m$ even}} \frac{\Bx_{1,1}(x^m)}{m}
 \biggr)  \\
 \BB &= \frac{x}{6}\bigl(C_{1,1}(x)^3 + 3C_{1,1}(x^2)C_2(x) + 2C_{1,1}(x^3)\bigr)\\
 C  &=  \tfrac12 C_{1,1}+\tfrac12 C_2\\
 E  &=  \frac{x}{2}\bigl(C_{1,1}(x)^3 +C_{1,1}(x^2) C_2(x)\bigr).
\end{align*}

These formulas simplify to
\begin{equation}
 \label{e-U2}
 U = C - \tfrac 13 x\bigl(C_{1,1}^3 -C_{1,1}(x^3)\bigr), \text{ with $C=\tfrac12(C_{1,1} + C_2)$},
\end{equation}
where
\begin{equation*}
 C_{1,1}= \exp\biggl(\sum_{m=1}^\infty \frac{x^m}{m}C_{1,1}(x^m)^2 \biggr)
\end{equation*}
and
\begin{equation*}
 C_2 =\exp\biggl(\sum_{\text{$m$ odd}}  \frac{x^m}{m}C_{1,1}(x^{2m})
 + \sum_{\text{$m$ even}}  \frac{x^m}{m}C_{1,1}(x^m)^2
 \biggr).
\end{equation*}

We can also  interpret  \eqref{e-U2} by a dissimilarity theorem, and this approach could be used to give a shorter self-contained derivation of the generating function for 2-trees. It is not hard to see that~$C$ counts (unlabeled) 2-trees rooted at an edge (by the number of triangles) and that $C_{1,1}$ counts 2-trees rooted at a directed edge.

For any (labeled) 2-tree $T$, a \emph{directed triangle} of $T$ is a cyclic orientation of a triangle of $T$. A \emph{symmetry triangle} is a directed triangle whose directed edges are all in the same orbit under the automorphism group of $T$.  A \emph{nonsymmetry triangle} is a directed triangle that is not a symmetry triangle. The generating function for 2-trees rooted at a symmetry triangle is $xC_{1,1}(x^3)$,  so the generating function for 2-trees rooted at a  nonsymmetry triangle is $\tfrac13 x(C_{1,1}^3 -C_{1,1}(x^3))$.  (See Figure \ref{f-triangle}.)
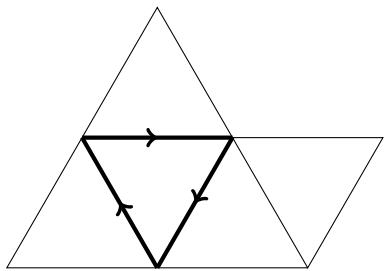
\begin{figure}[htbp] 
 \centering
 \begin{tikzpicture}
   \coordinate (a);
   \path [graphedge] (a) to ++(180:2) coordinate (b) to ++(180:2) coordinate (c) to ++(60:2) coordinate (d) to ++(60:2) coordinate (e) to ++(-60:2) coordinate (f) to (a) to ++(60:2) coordinate (g) to (f);
   \path [rootedge,->-] (d) to (f);
   \path [rootedge,->-] (f) to (b);
   \path [rootedge,->-] (b) to (d);
 \end{tikzpicture}

 \caption{A nonsymmetry-triangle-rooted 2-tree}
 \label{f-triangle}
\end{figure}

Then just as Otter's formula \eqref{e-U1} follows from Lemma \ref{l-diss1},  formula \eqref{e-U2} follows from a  \dc\ theorem for 2-trees:

\begin{lemma}
 \label{e-2diss}
 The number of dissimilar edges of a $2$-tree is one more than the number of dissimilar nonsymmetry triangles.
\end{lemma}

\begin{proof}
 Let $T$ be a 2-tree. Then $T$ has a center which is either an edge or a triangle and is fixed by every automorphism of $T$.  Let $e$ be an edge of $T$ that is neither a center edge nor an edge of a center triangle.  We will associate to~$e$ a nonsymmetry triangle of $T$.

 There is a ``path" from $e$ to the center passing alternatingly through edges and triangles  that starts at $e$, enters a triangle $\Delta$ containing $e$ and then visits another edge $f$ of  $\Delta$ (which may be the center).
 Then we associate to $e$ the triangle $\Delta$ oriented so that edge $e$ is followed by edge $f$.
 (See Figure \ref{f-2-diss} for an example in which the center is the edge $f$.)

 \begin{figure}[htbp] 
   \centering
   \newsavebox{\rooteda}
   \newsavebox{\rootedb}
   \newsavebox{\directeda}
   \newsavebox{\directedb}

   \begin{lrbox}{\rooteda}
     \begin{tikzpicture}
       \coordinate (a);
       \path [graphedge] (a) to ++(180:2) coordinate (b) to ++(180:2) coordinate (c) to ++(60:2) coordinate (d) to ++(60:2) coordinate (e) to ++(-60:2) coordinate (f) to (a) to ++(60:2) coordinate (g) to (f) to node [above,pos=0.6] {$f$} (b) to (d) to (f);
       \path [rootedge] (a) to node [below] {$e$} (b);
     \end{tikzpicture}
   \end{lrbox}

   \begin{lrbox}{\directeda}
     \begin{tikzpicture}
       \path [graphedge] (a) to ++(180:2) coordinate (b) to ++(180:2) coordinate (c) to ++(60:2) coordinate (d) to ++(60:2) coordinate (e) to ++(-60:2) coordinate (f) to (a) to ++(60:2) coordinate (g) to (f) to (b) to (d) to (f);
       \path [rootedge, ->-] (a) to (b);
       \path [rootedge, ->-] (b) to (f);
       \path [rootedge, ->-] (f) to (a);
     \end{tikzpicture}
   \end{lrbox}

   \begin{lrbox}{\rootedb}
     \begin{tikzpicture}
       \coordinate (a);
       \path [graphedge] (a) to ++(180:2) coordinate (b) to ++(180:2) coordinate (c) to ++(60:2) coordinate (d) to ++(60:2) coordinate (e) to ++(-60:2) coordinate (f) to (a) to ++(60:2) coordinate (g) to (f) to node [above,pos=0.6] {$f$} (b) to (d) to (f);
       \path [rootedge] (a) to node [above=3pt,pos=0.35] {$e$} (f);
     \end{tikzpicture}
   \end{lrbox}

   \begin{lrbox}{\directedb}
     \begin{tikzpicture}
       \path [graphedge] (a) to ++(180:2) coordinate (b) to ++(180:2) coordinate (c) to ++(60:2) coordinate (d) to ++(60:2) coordinate (e) to ++(-60:2) coordinate (f) to (a) to ++(60:2) coordinate (g) to (f) to (b) to (d) to (f);
       \path [rootedge, ->-] (b) to (a);
       \path [rootedge, ->-] (f) to (b);
       \path [rootedge, ->-] (a) to (f);
     \end{tikzpicture}
   \end{lrbox}

   \begin{tikzpicture}
     \node (rooted) at (0,0) {\usebox{\rooteda}};
     \node (directed) at (8,0) {\usebox{\directeda}};

     \draw [->, ultra thick] (rooted) to (directed);
   \end{tikzpicture}

   \begin{tikzpicture}
     \node (rooted) at (0,0) {\usebox{\rootedb}};
     \node (directed) at (8,0) {\usebox{\directedb}};

     \draw [->, ultra thick] (rooted) to (directed);
   \end{tikzpicture}

   \caption{Edges to directed triangles (with centers labeled $f$)}
   \label{f-2-diss}
 \end{figure}
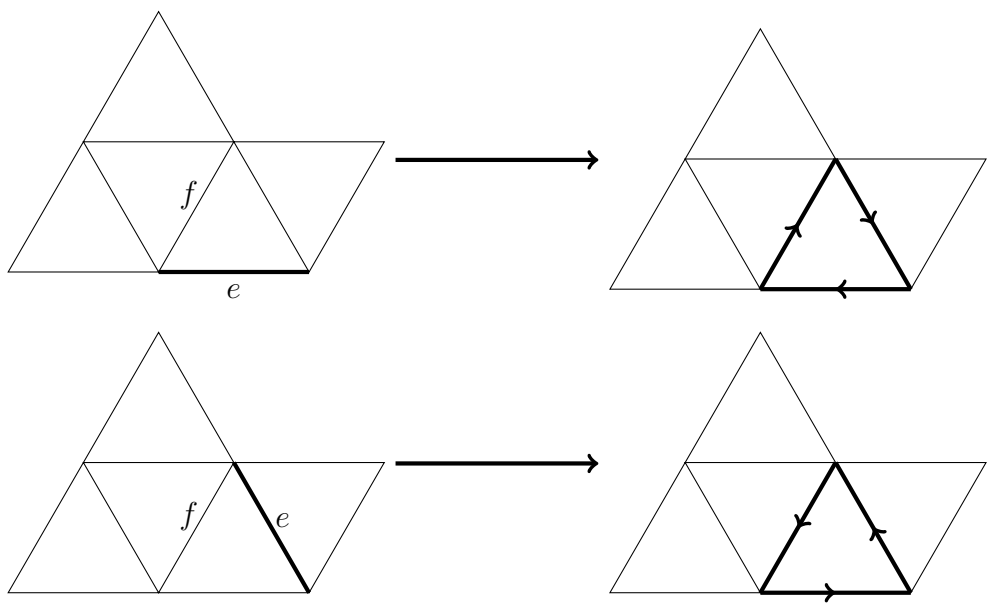
 The only unpaired edges and directed triangles  are center edges, edges of center triangles, and center triangles.
 We consider four  cases: (i) the center is an edge; (ii) the center is a triangle in which all three edges are similar; (iii) the center is a triangle with two dissimilar edges; and (iv) the center is a triangle with three dissimilar edges. In each case  the number of dissimilar unpaired edges is one more than the number of dissimilar unpaired nonsymmetry triangles:  In case (i) there is one  unpaired edge and no unpaired directed triangle; in case (ii) there are three unpaired edges, all of which are similar, and thus one dissimilar unpaired edge, and no unpaired nonsymmetry triangles; in case (iii) there are two dissimilar unpaired edges and two unpaired nonsymmetry triangles which are similar; and in case (iv) there are three dissimilar unpaired edges and two dissimilar unpaired directed triangles.
\end{proof}

%
%

Harary and Palmer \cite{h-p} counted 2-trees with the help of the following \dc\ theorem: In a 2-tree, let $q^*$ be the number of dissimilar edges, let $r^*$ be the number of dissimilar triangles, let $s_1$ be the number of dissimilar triangles with two similar edges, let $s_2$ be the number of dissimilar triangles with all three edges similar, and let $s=s_1+2s_2$. Then $q^* + s - 2r^* =1$.

It is not difficult to see that the number of dissimilar nonsymmetry triangles in a 2-tree is $2r^* -s$, so Harary and Palmer's \dc\ theorem is equivalent to ours. They applied their \dc\ theorem to derive a more complicated formula for the generating function for 2-trees:
\begin{equation*}
 L(x)+s_1(x)+2s_2(x) -2\triangle(x),
\end{equation*}
which is equal to our $U(x) -1$,
where in our notation,
$L(x)=C(x)-1$, $s_1(x)=x\bigl(C_{1,1}(x^2)C_2(x) - C_2(x^3)\bigr)$, $s_2(x) = xC(x^3)$,
and $\triangle(x)=\BB(x)$. They expressed these series in terms of auxiliary series $M_1(x)$, $N_1(x)$, $M(x)$, and $N(x)$, which in our notation are $M_1(x) = xC_{1,1}(x^2)$,
$N_1(x) =\tfrac12 x\bigl( C_{1,1}(x)^2 - C_{1,1}(x^2)\bigr)$,
$M(x) = C_2(x) -1$, and $N(x) = \tfrac12\bigl(C_{1,1}(x) - C_2(x)\bigr)$.

\bigskip

There are  formulas similar to \eqref{e-U1} and \eqref{e-U2} for all $k$, which we can find by expressing $U$ in terms of the $C_\mu$.
For $k=3$ the formula is
\begin{multline*}
 \quad U = C
 -x\bigl(\tfrac18C_{1,1,1}(x)^4 +\tfrac14C_{1,1,1}(x^2)C_{2,1}(x)^2 - \tfrac18 C_{1,1,1}(x^2)^2-\tfrac14 C_{1,1,1}(x^4)\bigr)
 \quad
\end{multline*}
and for $k=4$ it is
\begin{multline*}
 \qquad
 U=C
 -x\bigl(\tfrac1{30} C_{1,1,1,1}(x)^5 + \tfrac16 C_{1,1,1,1}(x^3) C_{3,1}(x)^2 +\tfrac16 C_{1,1,1,1}(x^2) C_{2,1,1}(x)^3 \\
 -\tfrac 16 C_{2,1,1}(x^3) C_{3,1}(x^2)
 -\tfrac 15 C_{1,1,1,1}(x^5)\bigr).
 \qquad
\end{multline*}
Although \eqref{e-U2}, as noted earlier, can be used to give a shorter proof of the formula for $2$-trees, it does not seem likely that these formulas can be used to simplify the counting of $k$-trees for $k>2$.

We note that a dissimilarity characteristic theorem for 3-trees, analogous to that of Harary and Palmer \cite{h-p} for 2-trees, has been given by Liang and Liu \cite{liang-liu}, but they did not use it to count 3-trees.
%

\section*{Acknowledgments}
The authors wish to thank two anonymous referees for their helpful suggestions.

\bibliographystyle{amsplain}
\bibliography{sources}

\end{document}